\documentclass[12pt]{amsart}
\usepackage{amsmath, amssymb, amsthm, amsfonts, mathrsfs}
\usepackage{cite}
\usepackage{amscd}
\usepackage{url}

\newtheorem{prop}{Proposition}
\newtheorem{defn}{Definition}
\newtheorem{lemma}{Lemma}

\newtheorem{theorem}[prop]{Theorem}

\theoremstyle{definition}

\newcommand{\ds}[1]{\ensuremath{\displaystyle{#1}}}

\newcommand{\ess}[1]{\hat{#1}}
\newcommand{\escortdist}{\ensuremath{\hat{\phi}}}
\newcommand{\philog}{\ensuremath{\log_{\phi}}}
\newcommand{\phiexp}{\ensuremath{\exp_{\phi}}}
\newcommand{\phimean}[2]{\mathbb{E}^{\phi}_{#1}\left[#2\right]}
\newcommand{\psimean}[2]{\mathbb{E}^{\psi}_{#1}\left[#2\right]}
\newcommand{\phivar}[2]{\text{Var}^{\phi}_{#1}\left[#2\right]}
\newcommand{\q}{\ensuremath{q}}
\newcommand{\qex}[2]{\mathbb{E}^{\q}_{#1}\left[#2\right]}

\begin{document}
\title{Escort Evolutionary Game Theory}
\author{Marc Harper}
\address{University of California Los Angeles}
\email{marcharper@ucla.edu} 
\date{\today}

\begin{abstract}
A family of replicator-like dynamics, called the escort replicator dynamic, is constructed using information-geometric concepts and generalized information diverenges from statistical thermodynamics. Lyapunov functions and escort generalizations of basic concepts and constructions in evolutionary game theory are given, such as an escorted Fisher's Fundamental theorem and generalizations of the Shahshahani geometry.
\end{abstract}

\maketitle

\section{Introduction}

Recent interest in generalized entropies and approaches to information theory include evolutionary algorithms using new selection mechanisms, such as generalized exponential distributions in place of Boltzmann distributions\cite{Dukkipati04, Dukkipati05} and the use of generalized logarithms and information measures in dynamical systems\cite{Martinez08}, statistical physics\cite{Tsallis88}, information theory \cite{Borland98}, \cite{Naudts02, Naudts04, Naudts08}, and complex network theory\cite{Turner07}. This paper brings these generalizations to evolutionary dynamics, yielding a class of dynamics containing the replicator dynamic and the orthogonal projection dynamic.

Information-theoretic approaches have produced models of adpated behavior similar to, and more general than, the replicator equation in certain ways.\cite{Sato05} Information geometry yields fundamental connections between evolutionary stability and information theory, generating the replicator equation from information divergences as gradient dynamics\cite{Amari95}. Statistical thermodynamics describes generalized information divergences\cite{Naudts08} which this paper uses to generate analogs of the replicator dynamic. Lyapunov functions are constructed from generalized information divergences.

The motivation for these constructions comes from the fact that while Fisher information and the Kullback-Liebler divergences have nice uniqueness properties in some settings, they are not the only way to measure information. Many other divergences have been studied in other contexts and a natural question is to ask what differences arise in models of natural selection that measure information differently.

The methods in this work yield several interesting specific examples that handle non-linear fitness landscapes, compute integrals of motion, and tightly summarize known results for the replicator and projection dynamics, including a single formula for the Lyapunov functions that manages to capture the Euclidean distance and the Kullback-Lieber information divergence, and many others between.

\section{Statistical Definitions}

This section describes the necessary functions and definitions to define the escort replicator equation. The following definitions are due to or adapted from contemporary work in statistical thermodynamics and information geometry\cite{Naudts04}.

Assume a given function $\phi$, called an \emph{escort}, that is strictly positive on $(0, 1)$. An escort induces a mapping of the set of categorical probability distributions (the simplex) into itself by applying the escort to each coordinate and normalizing. A common escort is $\phi(x) = x^{q}$, which yields constructions of the Tsallis and R\'{e}nyi type. This is a $q$-deformation as the limit $q \to 1$ recovers the distribution. All escorts in this paper will be non-decreasing.

Throughout, a function of a single variable applied to a vector is to be interpreted as applying to each coordinate, i.e. if $x = (x_1, \ldots, x_n)$ then $\log(x) = (\log(x_1), \ldots, \log(x_n))$. This abuse of notation will be clear from context and avoids excessive notation.

\begin{defn}[Partition Function, Escort Distribution]
For a distribution $x = (x_1, x_2, \ldots, x_n)$, define the \emph{partition function}
\[ Z_{\phi}(x) = \sum_{i=1}^{n}{ \phi(x_i)},\]
and the \emph{escort distribution}
\[ \escortdist(x) = \frac{1}{Z_{\phi}(x)} (\phi(x_1), \phi(x_2), \ldots, \phi(x_n)) = \frac{\phi(x)}{Z_{\phi}(x)}. \]
\end{defn}

Notation for a generalized mean is convenient. The escort expectation is the expectation taken with respect to the escort distribution in place of the original distribution.

\begin{defn}[Escort Expectation]
For a distribution $x = (x_1, x_2, \ldots, x_n)$ and a vector $f$ define the \emph{escort expectation}
\[ \phimean{x}{f} = \mathbb{E}_{\escortdist(x)}{\left[f\right]} = \frac{1}{Z_{\phi}(x)} \sum_{i=1}^{n}{\phi{(x_i)} f_i} = \frac{\phi(x) \cdot f}{Z_{\phi}(x)} = \escortdist(x) \cdot f. \]
\end{defn}

The escort function has various effects depending on its form. In the case of a constant escort, all distributions are mapped to the uniform distribution and the mean is just the average. If the escort function does not have the property that $\phi(0) = 0$, such as for $\phi(x) = e^x$, then the escort changes an event with probability $x_i=0$ to nonzero probability, affecting the computation of the mean if the escort distribution is used. For such an escort, The induced endomorphism maps the boundary of the simplex into the interior of the simplex.

Escort information divergences are defined by generalizing the natural logarithm using an escort function.

\begin{defn}[Escort Logarithm]
Define the escort logarithm
\[ \philog(x) = \int_{1}^{x}{\frac{1}{\phi{(v)}} \, dv} \]
\end{defn}

For example, for the function $\phi(x) = x^q$, the escort logarithm is
\[ \philog (x) = \frac{x^{1-q} - 1}{1-q}.\] The limit $q \to 1$ recovers the natural logarithm. A generalization of the logistic map of dynamical systems using this logarithm is given in \cite{Martinez08}.

The escort logarithm shares several properties with the natural logarithm: it is negative and increasing on $(0,1)$ and concave on $(0,1)$ if $\phi$ is strictly increasing. Define the function $\phiexp$ to be the inverse function of $\philog$.

\begin{defn}[Escort Divergences]
Define the escort divergence
\[ D_{\phi}(x || y) = \sum_{i=1}^{n}{\int_{y_i}^{x_i}{ \philog{(u)} - \philog{(y_i)} \, du} }.\]
\end{defn}

Since the logarithms are increasing on $(0,1)$, this divergence satisifies the usual properties of an information divergence on the simplex: $D(x||y) > 0$ if $x \neq y$ and $D(x||x) = 0$. The identity function $\phi(x) = x$ generates the usual logarithm and exponential with the Kullback-Liebler divergence. Setting $\phi(x) = x^q$ generates divergences similar to Tsallis and R\'{e}nyi divergences, and the $\alpha$-divergence of information geometry\cite{Dukkipati05-2, Amari93}.

\section{Geometry and Dynamics}

Define the escort metric \[ g^{\phi}_{ij}(x) = \frac{1}{\phi(x_i)} \delta_{ij}\] on the simplex. This is a Riemannian metric since the escort $\phi$ is strictly positive and so the metric is positive definite. The metric may be obtained as the Hessian of the escort divergence. The identity function $\phi(x) = x$ generates the Fisher information metric, also known as the Shahshahani metric in evolutionary game theory. Denote the simplex with the escort metric as the escort manifold.

%
%
It is known that the replicator equation is the gradient flow of the Shahshahani metric, with the right hand side of the equation a gradient with respect to the Shahshahani metric, if the fitness landscape is a Euclidean metric\cite{Shahshahani79, Hofbauer98}. This is also the case for the escort metric. See \cite{Hofbauer98} for a similar computation for the Shahshahani metric.

\begin{prop}
Let $\hat{f}_{i} = \phi(x_i) \left( f_i(x) - \phimean{x}{f(x)}  \right)$; $\hat{f}$ is a gradient with respect to the  escort metric if $f$ is a Euclidean gradient.
\end{prop}

\begin{proof}
For $\hat{f}(x)$ to be in the tangent space of the simplex requires that $\sum_{i}{ \hat{f}_{i}(x)} = 0$.
\begin{align*}
\sum_{i}{ \hat{f}_{i}(x) } &= \sum_{i}{\phi(x_{i}) \left( f_i(x) - \phimean{x}{f(x)} \right)}\\
&= \sum_{i}{\phi(x_i) f_i(x)} - \sum_{i}{ \phi(x_{i}) \left( \phimean{x}{f(x)} \right)}\\
&= \sum_{i}{\phi(x_i) f_i(x)} - \sum_{j}{ \phi(x_{j}) f_j(x)} = 0
\end{align*}

To see that $\hat{f}$ is a gradient with respect to the escort metric, recognize that the gradient is defined uniquely by the relation $<\text{grad} \, f, z>_{x} = D_x V(z)$, where $f$ is the Euclidean gradient of $V$. Verify that $\hat{f}$ has this property as follows:
\begin{align*}
<\hat{f}, z>_{x} &= \sum_{i}{ \frac{1}{\phi(x_i)} \ess{f}_{i}(x) z_i}\\
&= \sum_{i}{ \frac{1}{\phi(x_i)} \phi(x_i) \left( f_i(x) - \phimean{x}{f(x)} \right)  z_i}\\
&= \sum_{i}{ f_i(x) z_i} - \sum_{i}{z_i \left(\phimean{x}{f(x)} \right)}\\
&= \sum_{i}{ f_i(x) z_i} - \left(\sum_{i}{z_i}\right) \left(\phimean{x}{f(x)} \right)\\
&= \sum_{i}{ f_i(x) z_i} = \sum_{i}{ \frac{\partial V}{\partial x_i}(x) z_i} = D_x V(z)
\end{align*}
\end{proof}

Define the escort replicator equation as
\begin{equation}\label{escort_dynamic}
\dot{x}_i = \phi(x_i) \left( f_i(x) - \phimean{x}{f(x)}  \right)
\end{equation}

The escort dynamic shares many properties with the replicator dynamic. For instance, a Nash equilibrium of the fitness landscape $f$ is a rest point of the escort dynamic. Indeed, if $\ess{x}$ is a Nash equilibrium, there is a constant $c$ such $f(\ess{x})_i = c$ for all $i \in \text{supp}(\ess{x})$, from which it follows that $\phimean{\ess{x}}{f(\ess{x})} = c$. This gives a rest point of the dynamic since $\dot{x}_i = \phi(\ess{x}_i) \left( f_i(\ess{x}) - \phimean{\ess{x}}{f(\ess{x})} \right) = \phi(\ess{x_i})(c-c) = 0$. The next section will show that several well-known results for the replicator equation naturally generalize to the escort dynamic.

\section{Generalizations of Two Fundamental Results}

A generalized version of Kimura's Maximal Principle follows immediately from the fact that the escort dynamic is a gradient when the landscape is a Euclidean gradient. A simple calculation gives an analog of Fisher's Fundamental Theorem of Natural Selection (FFT). See \cite{Hofbauer98} for a similar calculation for the replicator dynamic.

\begin{theorem}[Escort FFT]
Suppose the landscape $f$ is the Euclidean gradient of a potential function $V$. Then
\[ \frac{d}{dt}V(x) = Z_{\phi}(x) \phivar{x}{f(x)}\]
\end{theorem}

\begin{proof}
\begin{align*}
\dot{V}(x) &= D_x V(\dot{x}) = <\ess{f}_{\phi}(x), \ess{f}_{\phi}(x)>\\
&= \sum_{i=1}^{n}{\frac{1}{\phi(x_i)} [\phi(x_i)(f_i(x) - \phimean{x}{f(x)})]^2}\\
&= Z_{\phi}(x) \phimean{x}{(f_i(x) - \phimean{x}{f(x)})^2} = Z_{\phi}(x) Var_{x}^{\phi}[f(x)]
\end{align*}
\end{proof}

The Kullback-Liebler divergence gives a local Lyapunov function for the replicator equation, a result which generalizes to the escort dynamic. See \cite{Sato05} for adaptive dynamics which also have a Lyapunov function given by information divergence. A point $\ess{x}$ is an evolutionarily stable state (ESS) if for all $x$ in some neighborhood of $\ess{x}$, $\ess{x} \cdot f(x) > x \cdot f(x)$. 

%

\begin{theorem}\label{lyapunov}
A state $\ess{x}$ is an ESS if and only if the escort divergence $\tilde{D}_{\phi}(\ess{x} || x)$ is a local Lyapunov function for the escort dynamic.
\end{theorem}
\begin{proof}
Let $V(x) = D_{\phi}(\ess{x} || x)$. $V$ is positive, and $V$ is zero if and only if $\ess{x} = x$. Taking the time derivative gives:
\begin{align*}
\dot{V} &= -\sum_{i}{ (\ess{x}_i - x_i) \frac{1}{\phi(x_i)} \dot{x}_i }\\
&= -\sum_{i}{ (\ess{x}_i - x_i) \frac{1}{\phi(x_i)} \phi(x_i)(f_i(x) - \phimean{x}{f(x)}) }\\
&= -\sum_{i}{ (\ess{x}_i - x_i)(f_i(x) - \phimean{x}{f(x)}) }\\
&= -(\ess{x} - x) \cdot f(x) < 0,
\end{align*}
where the last inequality is true if and only if the state $\ess{x}$ is an ESS. Hence the ESS is a local minimum and $V$ is a Lyapunov function for the dynamic.
\end{proof}


\section{Examples}

\subsection{Replicator dynamic}
Let $\phi(x) = x$ be the identity function. The $\philog$ and escort distributions are then just the ordinary logarithm and distribution, and the escort replicator equation is the replicator dynamic. The induced metric is the Shahshahani metric, and the induced divergence is the Kullback-Liebler divergence. 

\subsection{Replicator dynamic with selection intensity}

Let $\phi(x) = \beta x$, where $\beta > 0$ is the inverse temperature. The induced escort mapping is the identity because the parameter $\beta$ cancels. The escort logarithm is $\frac{1}{\beta} \log{x}$ and the induced escort dynamic is
\[\dot{x}_i = \beta x_i(f_i(x) - \bar{f}(x)),\]
a form in which can be derived using Fermi selection and stochastic differential equations (with parameter $\beta / 2$) \cite{Pacheco06}. The parameter $\beta$ can be interpreted as the intensity of selection. It affects the velocity of selection but not the trajectories.

In general, given any escort function $\psi$, a new escort $\phi(x) = \beta \psi(x)$ induces an escort dynamic with a leading intensity of selection factor. The trajectories are the same up to a change in velocity. Such dynamics can be arrived at by localizing other information divergences (see Discussion).

\subsection{$q$-deformed replicator dynamic}
Let $\phi(x) = x^q$. This generates the $q$-metric $x_{i}^{-\q} \delta_{ij}$ and the $q$-deformed replicator dynamic
\[ \dot{x}_i = x_{i}^{\q} \left( f_i(x) - \frac{ \sum_{j}{ x_{j}^{\q} f_j(x)} }{ \sum_{j}{x_{j}^{\q} }} \right) = x_{i}^{\q} \left( f_i(x) - \qex{x}{f(x)} \right). \]
The limiting case $q \to 1$ yields the replicator equation, as is usual for $q$-deformations. The $q$-deformed replicator dynamic is formally similar, but not identical, to the dynamic generated by the geometry of escort parameters, viewing $q$ as parameter, derived in \cite{Abe03}. For detail regarding the escort logarithms and Tsallis entropy see \cite{Furuichi09}. For a detailed exposition of the information geometry in this case see \cite{Ohara07}. The case $q=0$ is the orthogonal projection dynamic \cite{Joosten2008}.

The following example for $q=2$ gives an analog of the zero-sum property for a nonlinear fitness landscape and derives an integral of motion from the escort logarithm.

\subsubsection{Poincare Dynamic and a Constant of Motion} Consider the case $\phi(x) = x^2$ for $(q=2)$. This metric is called the Poincar\'{e} metric\cite{Horie08}. 
Call the associated dynamic the Poincar\'{e} dynamic.

Let us consider the following explicit example, adapted from an example for the replication equation in\cite{Hofbauer98}. For the replicator equation with a fitness landscape given by a matrix $A$, so that $f(x) = Ax$, the game matrix $A$ is called zero sum if $a_{ij} = -a_{ji}$. This yields a mean fitness $\bar{f}(x) = x \cdot Ax = 0$. For instance, consider the rock-scissors-paper matrix
\begin{center}
$A = \begin{bmatrix}
0 & 1 & -1 \\
-1 & 0 & 1 \\
1 & -1 & 0 \\
\end{bmatrix}$.
\end{center}
This gives a zero-sum game with fitness landscape $f(x) = (x_2 - x_3, x_3 - x_1, x_1 - x_2)$ and internal rest point $\ess{x} = (1/3, 1/3, 1/3)$ (that is not an ESS). The dynamic has an integral of motion given by $x_1 x_2 x_3$.

For the Poincare dynamic, it is not the case that $\phimean{x}{f(x)} = 0$, but the quadratic fitness landscape $f(x) = A \phi(x) = (x_2^2 - x_3^2, x_3^2 - x_1^2, x_1^2 - x_2^2)$ does have zero escort mean. This property is an analog of zero-sum, having a similar effect on the Poincare dynamic which now takes the form $\dot{x}_i = x_i^2 f_i(x)$. The state $\ess{x} = (1/3, 1/3, 1/3)$ is still an interior equilibrium but not an ESS, so Theorem \ref{lyapunov} does not apply. In this case the cross entropy is constant along trajectories of the dynamic. Using the fact that $\philog(x) = 1 - \frac{1}{x}$, we have an integral of motion for the Poincar\'{e} dynamic given by
\begin{equation}\label{integral}
I = \sum_{i}{\ess{x_i} \philog{x_i}} = 1 - \frac{1}{3}\left(\frac{1}{x_1} + \frac{1}{x_2} + \frac{1}{x_3}\right), 
\end{equation}
which is evident because
\[
\frac{d}{dt}\left( \frac{1}{x_1} + \frac{1}{x_2} + \frac{1}{x_3}\right) = -\left( \frac{\dot{x_1}}{x_1^2} + \frac{\dot{x_2}}{x_2^2} + \frac{\dot{x_3}}{x_3^3}\right) = 0. \]
Finally, note that the integral of motion in equation (\ref{integral}) for the original (non-linear landscape) example produces the constant of motion $1 / 3 \log{(x_1 x_2 x_3)}$. This quantity is constant if and only if $x_1 x_2 x_3$ is constant, so the equations are equivalent. The Kullback-Liebler divergence would suffice in the linear case of $\phi(x) = x$ as there is no difference between the divergence arising from the escort, and the naive divergence obtained by subsitituting $\philog$ for $\log$ in the KL-divergence. In the general case these two quantities differ and the naive divergence does not satisfy the divergence axioms.

There is a particularly nice interpretation of the information divergence in this context. The information divergence between two distributions can be interpreted as the information gain in the directed movement from the initial distribution to the final distribution. For a landscape with an ESS, along interior trajectories the possible information gain is decreasing toward a minimum at the ESS. This corresponds with an interpretation of the information divergence in inference as \emph{potential information}.\cite{Lee10}



\subsection{Orthogonal Projection Dynamic}
Let $\phi(x) = 1$, which gives the Euclidean metric. The resulting dynamic is the orthogonal projection dynamic,
\[ \dot{x}_i = f_i(x) - \frac{1}{n}\sum_{i=0}^{n}{f_i(x)},\]
where the expectation is the average. The Lyapunov function takes the form $\frac{1}{2}||x - \ess{x}||^2$, which captures the known result up to the factor of $1/2$.\cite{Sandholm08} This is the first non-forward invariant example of this section.

\subsection{Exponential Escort Dynamic}
Another example of a dynamic that is not forward-invariant is the dynamic generated by $\phi(x) = e^x$:
\[ \dot{x}_i = e^{x_i} \left( f_i(x) - \frac{ \sum_{j}{ e^{x_j} f_j(x)} }{ \sum_{j}{e^{x_j} }} \right). \]
For an explicit example, consider the fitness landscape $f(x) = e^{-x} = (e^{-x_1}, \ldots, e^{-x_n})$. After some simplification, the dynamic takes the form
\[ \dot{x}_i = 1 - \frac{n e^{x_i}}{\sum_{j}{e^{x_j}}}, \]
the right-hand side of which is nonzero for $x_i=0$, so this dynamic always resurrects extinct types. The barycenter $(\frac{1}{n}, \ldots, \frac{1}{n})$ is the unique rest point and is evolutionarily stable. A Lyapunov function is 
\[ \sum_{i}{e^{-x_i}\left(\frac{1}{n} + 1 - x_i\right)} - n e^{-1/n}.  \]


\subsection{Remark: Forward-Invariance}
The escort replicator equation is forward-invariant on the simplex, in the case of a non-constant fitness landscape, if and only if $\phi(0) = 0$, which is the case for the $q$-deformed replicator dynamic if $q > 0$, but not the case for the orthogonal projection dynamic, though the geometry is trivial. The exponential escort dynamic yields an example that is not forward-invariant with a non-trivial geometry.



\section{Discussion}

\subsection{Dynamics from Other Common Information Divergences}

A natural question to ask is whether other information divergences give interesting or known evolutionary dynamics. The escort dynamic is rather general in that it captures the localized form of many other commonly used information divergences. Two common classes of divergences, F-divergences and Bregman divergences, give dynamics captured by the escort dynamic. There are many other information divergences in the literature \cite{Lin91, Topose00, Endres03}, which appear to also yield dynamics in the escort class.
\subsubsection{F-divergences}
Let $F$ be a convex function on $(0, \infty)$ and consider the $F$-divergence 
\[D_{F}(x||y) = \sum_{i}{x_i F\left(\frac{y_i}{x_i}\right)}.\]
The $F$-divergence, when localized, produces a metric of the form $g_{ij}(x) = F''(1)\frac{1}{x_i}\delta_{ij}$, so these divergences yield dynamics that are special cases of the previous example \cite{Amari93}.
\subsubsection{Bregman and Burea-Rao divergences}
Bregman divergences are summations of terms of the form $\varphi(u,v) = f(u) - f(v) - f'(v)(u-v)$, which localize to $\ddot{\varphi}(v,v) = f''(v)$, corresponding to an escort given by the relation $\ds{f''(v) = \frac{1}{\phi(v)}}$. Similarly, Burbea-Rao divergences are in the escort class, up to a multiplicative constant\cite{Pardo03}. 


\subsection{Solutions of the Escort Replicator Equation}

In terms of the geometry, exponential families are normal coordinates on the manifold. The generalized exponential furnishes a formal solution to the escort dynamic. Let $v$ be a solution to $\dot{v}_i = f_i(x)$; then $x_i = \phiexp (v_i - G)$ is a solution to the escort dynamic where $\dot{G} = \phimean{x}{f(x)}$, which follows from the fact that $\dot{x}_1 + \cdots + \dot{x}_n = 0$. The following fact regarding the derivative of the escort exponential is needed for the proof of the solution (easily shown with implicit differentiation).
\begin{lemma}
\[ \frac{d}{dx}{\phiexp{x}} = \phi(\phiexp{x})\]
\end{lemma}
Simple diffferentiation shows that the escort exponential formally solves the escort dynamic.
\begin{align*}
\dot{x_i} &= \frac{d}{dt}{\phiexp (v_i - G)} = \phi(\phiexp (v_i - G)) (\dot{v_i} - \dot{G}) \\
 &= \phi(x_i) (f_i(x) - \phimean{x}{f(x)}).
\end{align*}
If the equations for $v$ and $G$ can be solved, an explicit solution can be given. This is analogous to the solution method for the replicator equation given in \cite{Nihat05}.

solved by eigenvalue methods, and if the equation for $G$ can be solved, will yield an explicit closed form.

%
%
%

\subsection{Velocity Transformations and Uniqueness of $q$-deformed Dynamics}
In some cases, the escort replicator equation can be transformed into the replicator equation with an altered fitness landscape. If $\phi$ is invertible and differentiable on $[0,1]$, with nonzero derivative, the transformation $y = \phi(x)$ translates the escort replicator equation to

\[ \dot{y_i} = \phi' (x_i) \phi(x_i)\left(f_i(\phi(x)) - \phimean{x}{f(\phi(x))} \right) =
\phi'(x_i) y_i\left(f_i(y) - \mathbb{E}_{y}\left[f(y)\right] \right), \]
so that for the fitness landscape $g(y) = f(\phi(x))$, we have that
\[ \dot{y_i} = \phi'(\phi^{-1}(y_i)) y_i \left(g(y) - \bar{g}(y) \right). \]
If $\ds{\phi'(\phi^{-1}(y_i))}$ is strictly positive \emph{and does not depend on $i$} then the last equation can be transformed into the replicator equation
\[ \dot{y_i} = y_i \left(g(y) - \bar{g}(y) \right), \]
by a strictly monotonic change in time scale. An example is $\phi(x) = \beta x$, which was shown above to produce a dynamic equivalent to the replicator equation.

To determine if the dynamics generated by $\phi(x)$ and $\psi(y)$ are distinct with respect to change of velocity, consider the equations
\[ \dot{x}_i = \phi(x_i) \left( f_i(x) - \phimean{x}{f(x)}  \right) \qquad \text{and}\]
\[\dot{y}_i = \phi(y_i) \left( f_i(y) - \psimean{y}{f(y)}  \right). \]
Assuming that $\psi$ is differentiable with differentiable invertible and that $\phi$ is differentiable, the transformation $\phi(x_i) = \psi(y_i)$ yields the equation
\begin{align*}
\dot{y}_i &= (\psi^{-1})'(\phi(x_i))\phi'(x_i)\dot{x}_i \\
\dot{y}_i &= (\psi^{-1})'(\phi(x_i))\phi'(x_i)\phi(x_i) \left( f_i(x) - \phimean{x}{f(x)}\right)  \\
\dot{y}_i &= (\psi^{-1})'(\phi(x_i))\phi'(x_i)\psi(y_i) \left( f_i(\phi^{-1}(\psi(y))) - \psimean{y}{f(\phi^{-1}(\psi(y)))}\right).
\end{align*}
If $(\psi^{-1})'(\phi(x_i))\phi'(x_i)$ does not depend on $i$, the $\phi$ dynamic is a $\psi$ dynamic (for an altered landscape) after a change in velocity. In the case of $\phi(x) = x^p$ and $\psi(y) = y^q$, this quantity becomes
\[ \frac{p}{q} x_i^{\frac{p-q}{q}},\]
which is independent of $i$ if and only if $p=q$.


\subsection{Vector-valued Escorts}

One can use vector-valued escort functions to define more general dynamics and all the preceeding results hold with little (if any) modification. E.g. the metric is now of the form $\ds{g_{ij}(x) = \frac{\delta_{ij}}{\psi_i(x)}}$, which is positive definite, and the induced dynamic is
\[ \dot{x}_i = \psi_i(x) \left( f_i(x) - \psimean{x}{f(x)}\right),\]
which differs because of the dependence on $i$ and the entire distribution in the leading term $\psi(x)$. The notational convention that $\psi(x)$ means $\psi$ is applied to each coordinate can be dropped for a vector-valued $\psi$ to obtain the definitions for escort mean and related quantities. If $\psi$ is identical in each coordinate, the leading factor can be eliminated with a change of velocity, since the dependence on $i$ is removed. (This is subtly different than for the escort dynamic, where $\phi(x_i)$ has dependence on the distribution argument but not the functional argument.) For a concrete example, the escort $\psi(x) = (\beta_1 x_1, \ldots, \beta_n x_n)$ yields a replicator-like dynamic with differential selection pressures on the individual types with Lyapunov function
\[D(x||y) = \sum_{i}{ \frac{1}{\beta_i} (y_i - x_i)} + \sum_{i}{\frac{x_i}{\beta_i} \log \left( \frac{x_i}{y_i} \right) }.\]

Examples are easy to generate with various choices of $\psi$, so as a Boltzmann-like dynamic using $\psi(x) = (e^{\beta f_1(x)}, \ldots, e^{\beta f_n(x)})$. With general $\psi$, care must be taken with the induced logarithms, as they differ for each coordinate and the defining integrals are now possibly vector-valued, and with the form of the information divergences. Nevertheless, many of the computations above hold with only minor changes (e.g. introduction of indicies for the escort), such as the generalization of Fisher's fundamental theorem.

\section{Conclusion}

Incorporating generalized information entropies from statistical thermodynamics yields interesting new dynamics and defines a class of dynamics that includes the replicator equation and the orthogonal projection dynamic. Some of the dynamics have interesting combinations of features, such as non-forward-invariance, while retaining nice informatic properties. By using constructions from information geometry, Lyapunov functions can be given for the class of dynamics arising from escorts that yield convex logarithms. Analogs of Fisher's fundamental theorem and other results follow naturally, yielding new evolutionary models.

\bibliography{ref}
\bibliographystyle{plain}

\end{document}